\documentclass{amsart}
\usepackage{amssymb, amsmath, amsfonts, amscd}

\input xypic

\theoremstyle{plain}
\newtheorem{theorem}{Theorem}[section]
\newtheorem{corollary}[theorem]{Corollary}
\newtheorem{lemma}[theorem]{Lemma}
\newtheorem{proposition}[theorem]{Proposition}

\newtheorem{example}[theorem]{Example}

\theoremstyle{definition}
\newtheorem{definition}[theorem]{Definition}

\theoremstyle{remark}

\numberwithin{equation}{theorem}

\newcommand{\F}{\mathcal{F}}

\newcommand{\I}{\mathcal{I}}

\newcommand{\Q}{Q}
\renewcommand{\L}{\mathcal{L}}
\newcommand{\E}{\mathcal{E}}

\renewcommand{\O}{\mathcal{O} }

\renewcommand{\Pr}{\mathcal{J} }

\renewcommand{\t}{\frac{1}{t} }

\newcommand{\SL}{\operatorname{SL}}

\renewcommand{\H}{\operatorname{H} }
\newcommand{\sym}{\operatorname{Sym}}

\newcommand{\R}{\operatorname{R} }
\newcommand{\U}{\operatorname{U}}

\newcommand{\lp}{\mathfrak{p}} 
\renewcommand{\lq}{\mathfrak{q}}
\newcommand{\lm}{\mathfrak{m}}
\renewcommand{\sl}{\mathfrak{sl}}

\renewcommand{\ln}{\mathfrak{n}}

\renewcommand{\t}{s}

\newcommand{\Spec}{\operatorname{Spec} }
\newcommand{\Funct}{\operatorname{Funct} }

\newcommand{\Pic}{\operatorname{Pic} }
\newcommand{\Mor}{\operatorname{Mor} }

\newcommand{\Proj}{\operatorname{Proj}}
\newcommand{\Diff}{\operatorname{Diff}}

\newcommand{\puI}{ \frac{\partial^{i_1}}{\partial_{u_1}}
  \frac{\partial^{i_2}}{\partial_{u_2}}\cdots
  \frac{\partial^{i_k}}{\partial_{u_k}}      }
\newcommand{\ptI}{ \frac{\partial^{i_0}}{\partial_{t_0}}
  \frac{\partial^{i_1}}{\partial_{t_1}}\cdots
  \frac{\partial^{i_n}}{\partial_{t_n}}      }
\newcommand{\pttI}{ \frac{\partial^{i_1}}{\partial_{t_1}}
  \frac{\partial^{i_2}}{\partial_{t_2}}\cdots
  \frac{\partial^{i_n}}{\partial_{t_n}}      }

\newcommand{\p}{\mathbb{P}}

\newcommand{\scha}{\underline{X} }
\newcommand{\schy}{\underline{Y} }

\newcommand{\sets}{\underline{Sets} }

\begin{document}

\title{On discriminants and incidence resolutions}

\author{Helge Maakestad}

\email{\text{h\_maakestad@hotmail.com} }
\thanks{Partially supported by a research scholarship from www.nav.no}


\date{1.4.2010}

\begin{abstract}  
  In this paper we study the rational points of the
  discriminant of a linear system on the projective line. We use this
  study to relate the discriminant $D_1(\O(d))$ to the classical
  discriminant of degree $d$ polynomials. We also study the incidence
  scheme $I_1(\phi)$ of an arbitrary morphism of finite rank locally
  free sheaves relative to an arbitrary quasi compact morphism of
  schemes. We
  prove $I_1(\phi)$ is a local complete intersection in general. We
  prove the
  existence of a complex - the incidence complex of $\phi$ - and prove
  it is a resolution of the ideal sheaf of $I_1(\phi)$ when $X$ is a
  Cohen-Macuaulay scheme. The aim of this study is to use it to study
  resolutions of the discriminant $D_1(\phi)$ of the morphism $\phi$.
\end{abstract}

\maketitle

\tableofcontents

\section{Introduction} 

In this paper we study the incidence complex of an arbitrary morphism
of sheaves relative to an arbitrary quasi compact morphism of schemes.
We prove the incidence complex of a morphism is a resolution of the ideal sheaf of
the incidence scheme when the initial scheme is irreducible
Cohen-Macaulay.
The incidence scheme of a morphim gives rise to the discriminant of a
morphism and the incidence complex of a morphism gives rise to the discriminant
double complex of a morphism. The aim of the study is to use the discriminant double
complex of a morphism to study resolutions of ideal sheaves of discriminants of
morphisms of sheaves.  
We also study the rational points of discriminants of linear systems
on the projective line and prove $D_1(\O(d))$ parametrize degree $d$
homogeneous polynomials in $x_0,x_1$ with multiple roots.

In the first section of the paper we study projective space and the
projective space bundle of a locally free finite rank sheaf. We relate
this to the notion of a representable functor and construct the
projective space bundle $\p(\E^*)$ of a locally free finite rank
$\O_X$-module $\E$ on any scheme $X$ using the Yoneda Lemma and the language of 
representable functors.

In section two of the paper we study the rational points of the
discriminant $D_l(\O(d))$ on the projective line over a field extension $L$ of some fixed
base field $K$ of characteristic zero. We prove $D_l(\O(d))(L)$
parametrize
degree $d$ homogeneous polynomials in $x_0,x_1$ with coefficients in
$L$ with a root in $\p^1_L$ of multiplicity at least $l+1$.

In section three of the paper we study the incidence scheme
$I_1(\phi)$ and discriminant scheme $D_1(\phi)$ of a morphism $\phi:u^*\E\rightarrow \F$
of locally free finite rank sheaves relative to a quasi compact
morphism $u:X\rightarrow S$
of schemes. 
When $\phi$ is surjective we prove $I_1(\phi)$ is a local complete
intersection (see Theorem \ref{lci}). We construct a complex - the
incidence complex of $\phi$ (see Definition \ref{incidence} ) - which is a candidate for a resolution of
the ideal sheaf of $I_1(\phi)$ and prove it is a resolution when $X$
is an irreducible Cohen-Macaulay scheme (see Corollary
\ref{resolution}).
The main aim of this study is to use the indicence resolution to
construct resolutions of ideal sheaves of discriminants of morphisms
of sheaves.

The discriminant of a morphism of sheaves is a simultaneous
generalization of the discriminant of a linear system on a smooth
projective scheme, the discriminant of a quasi compact morphism of
smooth schemes and the classical discriminant of degree $d$ polynomials. A
resolution of the ideal sheaf of the discriminant of a morphism of
sheaves would give a simultaneous construction of a resolution of the
ideal sheaf of the discriminant of a linear system on a smooth
projective scheme and the ideal sheaf of the discriminant 
of a quasi compact morphism of smooth schemes.

Much has been written about discriminants of linear systems on smooth
projective schemes and resultants (see \cite{gelfand} and \cite{jouanolou} for an overview of known
results and a reference list) and resolutions of ideal
sheaves of smooth projective schemes (see \cite{weyman}). The
technique we use in this paper where we realise the incidence scheme as
the zero scheme of a section of a locally free finite rank sheaf and
construct a Koszul complex giving rise to a resolution of its ideal
is a well known idea. The discriminant $D_1(\phi)$
of a morphism of locally free sheaves relative to a quasi compact
morphism of schemes introduced in this paper gives a unified construction of a large class of
discriminants appearing in algebraic geometry. I have not seen a
similar definition appearing in the litterature.

\section{On the tautological quotient bundle}

In this section we construct the tautological quotient bundle $\O(1)$
on the projective space bundle $\p(\E^*)$ where $E$ is any locally free
finite rank $\O_X$-module and $X$ is any scheme. 
We also define the tautological sequence
\[ \pi^*\E^*\rightarrow \O(1) \rightarrow .0\]
We relate the projective space bundle $\p(\E^*)$ to representable
functors and parameter spaces and prove some general properties needed
for the rest of the paper.

Let in the following $X$ be any scheme and let $\E$ be a locally free
$\O_X$-module of rank $d+1$.
and $\p=\p(\E^*)$ be the projective space bundle on $\E$ as defined in \cite{hartshorne}.
Let $\pi:\p(\E^*)\rightarrow X$ be the structure morphism. Let
$\pi_Y:Y\rightarrow X$ be any morphism of schemes.

\begin{example} On projective bundles and the Yoneda Lemma.\end{example}

Let $\scha$ be the category of schemes over $X$ and
$X$-morphisms and let $\sets$ be the category of sets and maps.
Let $\pi_Z:Z\rightarrow X$ be any scheme over $X$. 
Define the following functor:
\[ h: \scha \rightarrow \Funct^{op}(\scha, \sets) \]
by
\[ h_Z(U)=\Mor(U,Z) .\]
Here $\Funct^{op}(\scha,\sets)$ is the category of contravariant
functors
\[ F:\scha \rightarrow \sets \]
with natural transformations of functors as morphisms. One checks $h$ is a well
defined functor of categories. The \emph{Yoneda Lemma} states that $h$
is a fully faithful embedding of categories. This implies any natural
transformation of functors
\[ \eta:h_U \rightarrow h_V \]
comes from a uniques morphism $f:V\rightarrow U$. Hence $\eta=h(f)$.
Define the following functor
\[ P:\scha \rightarrow \sets \]
where
\[ P(Z)=\{\pi_Z^*\E^*\rightarrow^{p_\L} \L \rightarrow 0: \L\in \Pic(X) \}/ \cong .\]
Two morphisms $p_\L, p_{\L'}\in P(Z)$ are equivalent if there is an
isomorphism
\[ \phi:\L\rightarrow \L' \]
with all diagrams commutative. One checks $P$ defines a contravariant
functor between the two categories and it follows
\[ P\in \Funct^{op}(\scha,\sets) .\]
We say the functor $P$ is a \emph{representable functor} if it is in the image of
$h$: This means there is a scheme $\p\in \scha$ and an isomorphism
\[ \eta:P\rightarrow h_\p \]
of functors. An isomorphism of functors is a natural transformation
with an inverse. We say \emph{$P$ is represented by  $\p$}.

In the following we relate the projective space bundle $\p(\E^*)$
to representable functors and parameter spaces. The result is well
known but we include it for completeness.

\begin{proposition} \label{representable} Let $\p=\p(\E^*)$ be projective
  space bundle on $\E$. There is an isomorphism of functors
\[ \eta: P\rightarrow h_\p \]
hence the functor $P$ is represented by $\p(\E^*)$.
\end{proposition}
\begin{proof}  
There is a sequence of locally free sheaves
\[ \pi^*\E^*\rightarrow \O(1)\rightarrow 0\]
on $\p(\E^*)$ where $\O(1)$ is a linebundle. It has the following
property:
There is a one-to-one correspondence between the set of morphisms
\[ \phi:Y\rightarrow \p(\E^*) \]
over $X$ and the set of short exact sequences
\[ \pi_Y^*\E^* \rightarrow \L \rightarrow 0  \]
with $\L\in \Pic(Y)$ and $\phi^*(\O(1))=\L$. For a proof of this fact see \cite{hartshorne},
Proposition II.7.12. 
We want to use this result to define a natural transformation
\[ \eta: P\rightarrow h_\p \]
of functors.
Assume $\pi_Z:Z\rightarrow X$ is a scheme over $X$ and let 
\[ \pi_X^*\E\rightarrow \L \rightarrow 0 \]
be an element in $P(X)$. It follows by \cite{hartshorne}, Proposition II.7.12 
there is a unique morphism of $X$-schemes $\phi_\L:Z\rightarrow \p(\E^*)$
with $\phi^*(\O(1))=\L$. Define $\eta(Z)(\L)=\phi_\L \in
\Mor(Z,\p(\E^*))=h_\p(Z)$. 
We check $\eta$ defines a natural transformation of functors: Assume
\[ f:Y\rightarrow Z \]
is a morphism over $X$. We want to prove the following diagram
commutes:
\[
\diagram P(Y) \rto^{\eta(Y) } & \Mor(Y,\p) \\
         P(Z) \rto^{\eta(Z)}\uto^{f^*} & \Mor(Z,\p)\uto^{h(f)} 
\enddiagram .\]
Assume $\pi_Y^*\E^* \rightarrow \L \rightarrow 0$ is an element in
$P(Y)$. It follows $\L$ corresponds to a unique morphism over $X$
\[ \phi_\L: Y\rightarrow \p(\E^*) \]
with $\phi_\L^*(\O(1))=\L$. By pulling back via $f$ we get an exact
sequence on $Z$:
\[ \pi_Z^*\E^*\rightarrow f^*\L\rightarrow 0.\]
This sequence gives rise to a morphism
\[ \phi_{f^*\L}:Z\rightarrow \p(\E^*) \]
with
\[ \phi_{f^*\L}^*(\O(1))=f^*\L.\]
There is a diagram
\[
\diagram X \rto^{\phi_{f^*\L}} \dto^f & \p(\E^*) \\
         Z \urto^{\phi_\L} & 
\enddiagram \]
of morphisms of schemes.
It follows 
\[ (\phi_\L\circ f)^*(\O(1))=f^*(\phi_\L^*(\O(1)))=f^*\L=\phi_{f^*\L}^*(\O(1)) \]
hence by unicity it follows
\[ \phi_{f^*\L}=\phi_\L\circ f\]
hence the diagram commutes and $\eta$ is a natural transformation of functors.
Assume $\L,\L'\in P(Z)$. Let
$\eta_Z(\L)=\phi_\L $ and $\eta_Z(\L')=\phi_{\L'}$. Assume $\phi_\L=\phi_{\L'}$.
It follows 
\[ \L=\phi_\L^*(\O(1))=\phi_{\L'}^*(\O(1))=\L' \]
hence $\L=\L'$ and $\eta_Z$ is injective for all $Z$. Assume $\phi:Z\rightarrow
\p(\E^*)$ is a morphism. Let $\L=\phi^*(\O(1))\in P(Z)$. It follows
since $\phi$ is unique that $\eta_Z(\L)=\phi$ hence $\eta_Z$ is
surjective. It follows $\eta$ is an isomorphism of functors and the
Proposition is proved.
\end{proof}

\begin{definition} \label{tautological} The invertible sheaf $\O(1)$
  is  the \emph{tautological
    quotient bundle} on $\p(\E^*)$. The invertible sheaf $\O(-1)$ is
  the \emph{tautological subbundle} on $\p(\E^*)$.
 The exact sequence
\[ \E^*\otimes \O_{\p(\E^*)}\rightarrow \O(1)\rightarrow 0\]
is the \emph{tautological sequence}.
\end{definition}

As an application we will calculate the fiber $\pi^{-1}(\lp)$ of the
projection morphism $\pi:\p(\E^*)\rightarrow X$ using the Yoneda Lemma
and the tautological sequence:

\begin{lemma} \label{fiber} There is for every $\lp \in X$ a fiber diagram
\[
\diagram  \p(\E(\lp)^*)\cong \pi^{-1}(\lp)\rto^i \dto^{\tilde{\pi}} &
\p(\E^*) \dto^\pi \\
          \Spec(\kappa(\lp)) \rto^j & X
\enddiagram .\]
\end{lemma}
\begin{proof} 
Let $\lp \in X$ be a point and let
  $Y=\Spec(\kappa(\lp))$. Let $\phi:Y\rightarrow X$ be the canonical map.
Define the following functor
\[ P^\lp:\schy \rightarrow \sets \]
by
\[ P^\lp(Z,f)=\{f^*\E(\lp)^*\rightarrow \L\rightarrow 0 : \L \in
\Pic(Z)\} / \cong .\]
It follows from Proposition \ref{representable} $P^\lp$ is represented by the scheme $\p(\E(\lp)^*)$
parametrizing lines in $\E(\lp)$. Hence there is an isomorphism of
functors
\[ P^\lp \cong h_{\p(\E(\lp)^*)}.\]
Let $\pi:\p(\E^*)\rightarrow X$ be the canonical map and let $\pi^{-1}(\lp)=Y\times_X \p(\E^*)$ be the
fiber of $\pi$ at $\lp$. There is an isomorphism of functors
\[ h_{\pi^{-1}(\lp)}\cong h_Y\times_{h_X}h_\p .\]
We want to define an isomorphism
\[ \eta:P^\lp \rightarrow h_Y\times_{h_X}h_\p \]
of functors. 
Assume $(Z,f)$ is a scheme over $Y$ and let $f^*\E(\lp)^*\rightarrow
\L \rightarrow 0$ be an element of $P^\lp(Z,f)$. It follows there is
an equality $f^*\E(\lp)^*\cong (f\circ \phi)^*\E^*$. Hence $\L$
corresponds to a unique morphism $\phi_\L:Z\rightarrow \p(\E^*)$ with $\pi\circ
\phi_\L=f\circ \phi$. Define
\[ \eta(Z)(\L)=(f,\phi_\L)\in \Mor(Z,Y)\times_X \Mor(Z,\p).\]
One checks $\eta$ is an isomorphism of functors by exhibiting an
inverse natural transformation $\eta^{-1}: h_Y\times_{h_X}h_\p \rightarrow P^\lp$.
 It follows there is an isomorphism
\[ h_{\pi^{-1}(\lp)}\cong h_{\p(\E(\lp)^*)} \]
of functors. The claim of the Lemma now follows from the Yoneda Lemma
since $h$ is a fully faithful embedding of categories.
\end{proof}

When $X$ is any scheme over $K$ and $K\subseteq L$ is a field
extension we let $X(L)$ denote the set of $L$-rational points of
$X$. By definition
\[ X(L)=\{ \phi:\Spec(L)\rightarrow X \}.\]

\begin{lemma} \label{rational} Let $A$ be a commutative $K$-algebra and let $X=\Spec(A)$.
There is a one-to-one correspondence between the sets
\[ X(L)\cong \{\lp\subseteq A\text{ prime }: \kappa(\lp)\subseteq
L\text{ an extension} \} .\]
\end{lemma}
\begin{proof} Let $x\in X(L)$. It follows $x$ corresponds to a map
\[ \phi:A\rightarrow L \]
of rings. Let $\lp=ker(\phi)$. It follows $A/\lp\subseteq L$ is an
inclusion hence $A/\lp$ is a domain and $\lp$ is a prime ideal. It
follows we get an inclusion
\[ \kappa(\lp)=K(A/\lp)\subseteq L .\]
Hence we have a correspondence as claimed. 
This sets up a bijective correspondence and the Lemma follows.
\end{proof}

Let in the following $A=K$ be any field and $W=K\{e_0,..,e_d\}$.

\begin{corollary} \label{parametrize} Let $K\subseteq L$ be an extension of fields. There is a bijection of sets
\[ \p(W^*)(L)\cong \{ l\subseteq W\otimes_K L :l\text{ is a line.} \} \]
\end{corollary}
\begin{proof} Let $\p=\p(W^*)$ and let $\pi_L:\Spec(L)\rightarrow
  \Spec(K)$ be the natural map.
Consider the natural transformation $\eta:P\rightarrow  h_\p$ from Proposition \ref{representable}.
We get a bijection of sets
\[ \eta_{\Spec(L)}:P(\Spec(L))\rightarrow
h_\p(\Spec(L))=\Mor(\Spec(L),\p(W^*))=\p(W^*)(L) .\]
Assume 
\[ 0 \rightarrow l \rightarrow W\otimes_K L \]
is a rank one line in $W\otimes_K L$. Let 
\[ \pi_L^*W=(W\otimes_K L)^* \rightarrow l^* \rightarrow 0 \]
be its dual. By Proposition \ref{representable} we get a unique morphism
\[ \phi:\Spec(L)\rightarrow \p(W^*) \]
with 
\[ \phi^*(\O(1))=\O(1)(x)=l^*.\]
The morphism $\phi$ corresponds by Lemma \ref{rational} to a point $x\in
\p(W^*)(L)$ and an extension $\kappa(x)\subseteq L$ of fields. This
correspondence is one-to-one and the Corollary is proved.
\end{proof}

\begin{example} On projective space and parameter spaces. \end{example}

It follows the $K$-rational points $\p(W^*)(K)$ of the scheme
$\p(W^*)$ parametrize lines $l\subseteq W$: We get from Corollary
\ref{parametrize} a bijection of sets
\begin{align}
&\label{bijection} \phi: \p(W^*)(K)\cong \{l\subseteq W: l\text{ is a line.}\}
\end{align}
given as follows: Let $\{x\}=\Spec(K)\rightarrow \p(W^*)$ be a
$K$-rational point. The correspondence \ref{bijection} is given by
\[ \phi(x)=\O(-1)(x)\subseteq W.\]
Hence the scheme $\p(W^*)$ is a \emph{parameter space}.

Let $\pi:\p(W^*)\rightarrow X$ be the projection morphism where $W$ is
a finite rank locally free $A$-module and $X=\Spec(A)$. Let $\lp\in X$
be a prime ideal.
By Lemma \ref{fiber} it follows the fiber $\pi^{-1}\cong \p(W(\lp)^*)$ parametrize lines in the
$\kappa(\lp)$-vector space $W(\lp)$.
One checks the tautological sequence on $\p(W(\lp)^*)$ is the pull
back of the tautological sequence on $\p(W^*)$ via the canonical
morphism
\[ \p(W(\lp)^*)\rightarrow \p(W^*).\]
The morphism $\pi:\p(W^*)\rightarrow X$ is a locally trivial fibration
with fibers $\p^d_{\kappa(\lp)}\cong \Proj(\kappa(\lp)[y_0,..,y_d])$
for all points $\lp\in X$.

If $K$ is an algebraically closed field and $A$ a finitely generated
$K$-algebra the closed points $\lm$ of $X=\Spec(A)$ have residue field
$\kappa(\lm)\cong K$. It follows the fiber of $\pi$ at closed points
$\lm$ equals $\p^d_K$.

\section{Rational points of discriminants on the projective line}

In this section we study the rational points of the incidence scheme $I_l(\O(d))$ and
discriminant scheme $D_l(\O(d))$ where $\O(d)$ is a linebundle on the
projective line $\p^1_K$ over an arbitrary field extension $K\subseteq
L$. As a consequence we prove the discriminant $D_1(\O(d))$ equals the 
classical discriminant of degree $d$ polynomials.

Let $\O(1)$ be the tautological quotient bundle on $\p(V^*)$ where $V=K\{e_0,e_1\}$ and let
$\O(d)=\O(1)^{\otimes d}$. Let $V^*=K\{x_0,x_1\}$ where $x_i=e_i^*$.
Let $W=\H^0(\p(V^*),\O(d))$ and let $s_i=x_0^{d-i}x_1$. It follows
$s_0,..,s_d$ is a basis for $W$. Let $W^*$ have basis $y_0,..,y_d$
where $y_i=s_i^*$. Let $K\subseteq L$ be any field extension of $K$.
Let 
\[ \Delta:\p(V^*)\rightarrow \p(V^*)\times \p(V^*) \]
be the diagonal embedding and let $\I$ be the ideal of the diagonal.
Let 
\[ p,q:\p(V^*)\times \p(V^*)\rightarrow \p(V^*) \]
be the canonical projection maps and let $Y=\p(V^*)\times \p(V^*)$.

\begin{definition} Let 
\[ \Pr^l(\O(d))=p_*(\O_Y/\I^{l+1}\otimes q^*\O(d))) \]
be the \emph{$l$'th order jet bundle} of $\O(d)$.
\end{definition}
There is on $Y$ a short exact sequence of locally free sheaves
\[ 0\rightarrow \I^{l+1}\rightarrow \O_Y \rightarrow \O_{\Delta^l}
\rightarrow 0\]
and applying the functor $p_*(-\otimes q^*\O(d)) $ we get a long exact
sequence
of locally free sheaves
\[ 0\rightarrow p_*(\I^{l+1}\otimes q^*\O(d))\rightarrow
p_*q^*\O(d)\rightarrow \Pr^l(\O(d)) \rightarrow \]
\[ \R^1p_*(\I^{l+1}\otimes \O(d)) \rightarrow \cdots \]
Let $\p=\p(V^*)$.
There is by flat basechange an isomorphism
\[ p_*q^*\O(d)\cong \pi^*\pi_*\O(d)\cong \H^0(\p,\O(d))\otimes
\O_{\p} \]
of sheaves hence we get a morphism
\[ T^l:\H^0(\p,\O(d))\otimes \O_{\p}\rightarrow \Pr^l(\O(d)) \]
called the \emph{$l$'th Taylor morphism} of $\O(d)$. We will in the
following use the Taylor morphism and the tautological  subbundle
$\O(-1)$ to define the $l$'th incidence scheme $I_l(\O(d))$ and the $l$'th discriminant $D_l(\O(d))$. 

There is from Definition \ref{tautological} the \emph{tautological
  subbundle} on $\p(W^*)$:
\begin{align}
&\label{taut} 0\rightarrow \O(-1)\rightarrow W\otimes \O_{\p(W^*)} .
\end{align}

It has the following property: By the results of the previous section 
it follows projective space $\p(W^*)$
parametrize lines in the vector space $W$. This implies any
$K$-rational point $x\in \p(W^*)(K)$ corresponds uniquely to a line
$l_x\subseteq W$. The line $l_x$ is given by the tautological sequence
\ref{taut}: Take the fiber of \ref{taut} at $x$ and let
$l_x=\O(-1)(x)$. We get an inclusion
\[ l_x\subseteq (W\otimes \O_{\p(W^*)})(x)\cong W \]
of vector spaces. This correspondence sets up a bijection
\begin{align}
&\label{bij} x\in \p(W^*)(K)\cong \{ \text{lines }l_x\subseteq W\} .
\end{align}

The tautological sequence \ref{taut} is given by the sheafification of
the following sequence of $K[y_0,\cdots, y_d]$-modules:
\[ K[y_0,.., y_d](-1)\rightarrow K[y_0,.., y_d]\otimes W \]
\[ 1\rightarrow \sum y_i\otimes s_i .\]
One easily checks the sequence \ref{taut} gives rise to the bijection \ref{bij}.

There is a diagram of morphisms of schemes
\[
\diagram  \p(W^*)\times \p \rto^p \dto^q & \p \dto^\pi \\
            \p(W^*) \rto^\pi & \Spec(K) 
\enddiagram .
\]
Let $Y=\p(W^*)\times \p$. On $\p$ there is the Taylor morphism
\[ T^l:\H^0(\p,\O(d))\otimes \O_\p \rightarrow \Pr^l(\O(d)) \]
and on $\p(W^*)$ there is the tautological sequence
\[ \O(-1)\rightarrow \O_{\p(W^*)}\otimes W .\]
Pull these morphisms back to $Y$ to get the composed morphism
\[ \phi:\O(-1)_Y\rightarrow \H^0(\p(V^*),\O(d))\otimes \O_Y\rightarrow
\Pr^l(\O(d))_Y .\]
Let $Z(\phi)\subseteq \p(W^*)\times \p(V^*)$ denote the zero scheme of
the morphism $\phi$. By definition a point $\lp$ is in $Z(\phi)$ if
and only if $\phi(\lp)=0$.

\begin{definition} The scheme $I_l(\O(d))=Z(\phi)$ is the 
\emph{$l$'th incidence scheme } of $\O(d)$.
The direct image scheme $D_l(\O(d))=q(I_l(\O(d))$ is the \emph{$l$'th
  discriminant} of $\O(d)$.
\end{definition}

Since $I_l(\O(d))$ is a closed subscheme of a projective scheme and the projection $q$ is a
proper morphism it follows $D_l(\O(d))$ is a closed subscheme of $\p(W^*)$.

Let $K\subseteq L$ be a field extension. There is a bijection
\[ (\p(W^*)\times \p^1_K)(L)\cong \p(W^*)(L)\times \p^1_K(L) \]
of sets.

Assume $s\in \p(W^*)(L)$. Let its corresponding line
be $ l_s=\O(-1)(s) \subseteq W\otimes_K L$.

\begin{proposition} \label{Lrational} Let $Y=\p(W^*)\times \p^1_K$. There is a one-to-one correspondence of sets
\[ I_l(\O(d))(L)\cong \{(s,x)\in Y(L): T^l(x)(l_s)=0 \text{ in }\Pr^l(\O(d))(x).\} \]
\end{proposition}
\begin{proof} Consider the diagram
\[
\diagram \p(W^*)\times \p^1_K \rto^p \dto^q & \p^1_K \dto^\pi \\
        \p(W^*) \rto^\pi & \Spec(K)
\enddiagram .\]
We get a sequence of locally free sheaves on $Y$:
\[ \phi: p^*\O(-1)\rightarrow W\otimes \O_Y \rightarrow
q^*\Pr^l(\O(d)) \]
and $z\in I_l(\O(d))$ if and only if $\phi(z)=0$. Assume $z=(s,x)\in
Y(L)$. it follows $s\in \p(W^*)(L)$ and $x\in \p^1_K(L)$. We see that
$\phi(z)=0$ if and only if the composed map
\[ l_s=\O(-1)(s)\subseteq W\otimes_K L
\rightarrow^{T^l(x)}\Pr^l(\O(d))(x) \]
is zero.
This is if and only if $T^l(x)(l_s)=0$ and the Proposition is proved.
\end{proof}

\begin{corollary} The following holds: There is a one-to-one correspondence of sets
\[ D_l(\O(d))(L)\cong \{ s \in \p(W^*)(L): \text{ there is a }x\in
\p^1_K(L)\text{ with }T^l(x)(l_s)=0 \}.\]
\end{corollary}
\begin{proof} The Corollary follows directly from Proposition \ref{Lrational}.
\end{proof}

Let $t=x_1/x_0$ and let $f(t)\in L[t]_d$ be a degree $d$ polynomial. We let $\alpha_i$
denote the $i$'th coefficient of $f(t)$. Let 
\[ U_{ij}=D(y_i)\times D(x_j)\subseteq Y \]
be the basic open subset corresponding to $y_i$ and $x_j$. Let 
\[ I_l(\O(d))_{ij}=I_l(\O(d))\cap U_{ij} \]

Let $U_i=D(y_i)$ and let $u_j=y_j/y_i$ for $j=0,..,d$.
The map
\[ K[u_0,..,u_d]\frac{1}{y_i}\rightarrow K[u_0,..,u_d]\otimes W \]
looks as follows:
\[ \frac{1}{y_i}\rightarrow u_0s_0+\cdots +u_ds_d =\]
\[ u_0x_0^d+u_1x_0^{d-1}x_1+\cdots +u_dx_1^d= f(t)x_0^d \]
where
\[ f(t)=u_0+u_1t+\cdots 0+u_{i-1}t^{i-1}+t^i+u_{i+1}t^{i+1}+\cdots
+u_dt^d.\]
We get the following map
\[ \O(-1)_{U_{i0}}\rightarrow W\otimes \O_{U_{i0}} \]
\[ K[u_0,..,u_d][t]\frac{1}{y_i}\rightarrow K[u_0,..,u_d][t]\otimes W \]
defined by
\[ \frac{1}{y_i}\rightarrow f(t)x_0^d.\]
The Taylor map looks as follows:
\[ T^l(f(t)x_0^d)=f(t+dt)\otimes x_0^d= f(t)\otimes x_0^d+f'(t)dt\otimes x_0^d +\cdots +
\frac{f^{(l)}(t)}{l!}dt^l\otimes x_0^d.\]
It follows the ideal sheaf of $I_l(\O(d))_{i0}$ looks as follows:
\[ \I_{U_{i0}}=\{ f(t),f'(t),..,\frac{f^{(l)}(t)}{l!} \}. \]

Let $\frac{x_0}{x_1}=s$.
The map
\[ K[u_0,..,u_d]\frac{1}{y_i}\rightarrow K[u_0,..,u_d]\otimes W \]
looks as follows:
\[ \frac{1}{y_i}\rightarrow u_0s_0+\cdots +u_ds_d =\]
\[ u_0x_0^d+u_1x_0^{d-1}x_1+\cdots +u_dx_1^d= g(s)x_1^d \]
where
\[
g(s)x_1^d=(u_0(s)^d+u_1(s)^{d-1}+\cdots
+ u_d)x_1^d.\]
The Taylor map looks as follows:
\[ T^l(g(\t)x_1^d)=g(\t+d\t)\otimes x_1^d= g(\t)\otimes
x_1^d+g'(\t)d\t\otimes x_1^d +\cdots +
\frac{g^{(l)}(\t)}{l!}d\t^l\otimes x_1^d.\]
It follows the ideal sheaf of $I_l(\O(d))_{i1}$ looks as follows:
\[ \I_{U_{i1}}=\{ g(\t),g'(\t),..,\frac{g^{(l)}(\t)}{l!} \}. \]

Let $K\subseteq L$ be any field extension.
Let $W=\H^0(\p_K^1 ,\O(d))$ and let $f\in L\otimes _K W$ be a degree
$d$ homogeneous polynomial in $x_0,x_1$ with coefficients in $L$.

\begin{definition}
We say an element $(\alpha,\beta)\in \p^1_L$ is a \emph{root of
  $f(x_0,x_1)$ of multiplicity $\geq l+1$} if we may write
\[ f(x_0,x_1)=h(x_0,x_1)(\beta x_0-\alpha x_1)^{l+1}.\]
where $h(x_0,x_1)\in L\otimes_K W$.
\end{definition}

\begin{theorem} \label{Lrational} There is a bijection of sets
\[I_l(\O(d))(L)\cong \{ (f(x_0,x_1),(\alpha,\beta)): f(x_0,x_1)\in L\otimes_K
W, (\alpha, \beta)\in \p^1_L \} \]
where $(\alpha,\beta)$ is a root of $f$ with multiplicity $\geq l+1$.
\end{theorem}
\begin{proof} Let $Y=\p(W^*)\times \p^1_K$ and let
\[ \phi:\O(-1)_Y\rightarrow \Pr^l(\O(d))_Y \]
with $Z(\phi)=I_l(\O(d))$. Let $I_{ij}=Z(\phi)\cap U_{ij}$ with $U_{ij}=D(y_i)\times
D(x_j)$. Let $u_j=y_j/u_i$. It follows the coordinate ring on $I_{i0}$
equals
\[ A=K[u_0,..,u_d,t]/(f(t),..,f^{(l)}(t)) \]
with
\[ f(t)=u_0+u_1t+\cdots +u_dt^d.\]
The coordinate ring $B$ on $I_{i1}$ equals
\[B=K[u_0,..,u_d,s]/(g(s),..,g^{(l)}(s) \]
where
\[ g(s)=u_0s^d+\cdots +u_d.\] 
A point $x\in I_{i0}(L)$ corresponds bijectively to a morphism
\[ \psi:A\rightarrow L  .\]
Let $\psi(u_i)=\alpha_i$ and $\psi(t)=\beta$. Let 
\[ f_\alpha(t)=\alpha_0+\cdots +\alpha_dt^d.\]
It follows
\[ f_\alpha(\beta)=\cdots =f_\alpha^{(l)}(\beta)=0 \]
hence $f$ and $\beta$ gives rise to a pair $(\tilde{f},
(\tilde{\alpha},\tilde{\beta}) )$
with $\tilde{f}\in L\otimes_K W$ and $(\tilde{\alpha},\tilde{\beta})$
a root of $\tilde{f}$ of multiplicity $\geq l+1$. A similar argument
works when $x\in I_{i1}(L)$ and the Theorem is proved.
\end{proof}

Let $D_l(\O(d))_{i}=D_l(\O(d))\cap U_{i}$.

\begin{corollary} There is a bijection of sets
\[ D_l(\O(d))_{i}(L)\cong \{f(x_0,x_1): f(x_0,x_1)\in L\otimes_K W \} \]
such that $f(x_0,x_1)$ has a root $(\alpha,\beta)\in \p^1_L$ of
multipliticy $\geq l+1$.
\end{corollary}
\begin{proof}  Since $D_l(\O(d))=q(I_l(\O(d))$ the Corollary follows
  from Theorem \ref{Lrational}.
\end{proof}

It follows
\[ D_l(\O(d))(L)\subseteq \p(W^*)(L) \]
parametrize degree $d$ homogeneous polynomials
\[ f(x_0,x_1)\in L\otimes _K \H^0(\p(V^*),\O(d)) \]
with a root $z\in \p^1_L$ of multiplicity at least $l+1$. It follows we get a
filtration of sets
\[ D_d(\O(d))(L)\subseteq \cdots \subseteq D_1(\O(d))(L)\subseteq
\p(W^*)(L) \]
at the level of $L$-rational points. 
It follows the scheme $D_1(\O(d))$ is the scheme whose $L$-rational
points are homogeneous degree $d$ polynomials in $x_0,x_1$  with coefficients in $L$
with multiple roots. It follows $D_1(\O(d))$ equals the
\emph{classical discriminant of degree $d$ polynomials}. In a previous
paper on the subject (see \cite{maa1}) this result was proved using different methods.

\section{On incidence complexes for morphisms of locally free sheaves}

In this section we study the incidence complex of an arbitrary
morphism $\phi:u^*\E\rightarrow \F$ of locally free sheaves $\E,\F$
relative to an arbitrary quasi compact morphism $u:X\rightarrow S$ of
schemes. We prove the incidence complex is a resolution of the ideal
sheaf of the incidence scheme $I_1(\phi)$ when $X$ is a Cohen-Macaulay scheme.
We also define the discriminant double complex of $\phi$ using the incidence complex.

Let in the following $u:X\rightarrow S$ be an arbitrary quasi compact
morphism of schemes. Let $\E$ be a locally free 
$\O_S$-module of rank $e$ and let $\F$ be a locally free $\O_X$-module
of rank $f$. 
Let $\phi:u^*\E\rightarrow \F$ be a surjective morphism of
$\O_X$-modules. We get an exact sequence of locally free
$\O_X$-modules
\begin{align}\label{seq1}
 0\rightarrow \Q \rightarrow u^*\E \rightarrow^\phi \F \rightarrow 0
.
\end{align}
It follows $rk(\Q)=e-f$. Let $Y=\p(u^*\E^*)=\p(\E^*)\times_S X$ and
consider the following fiber diagram of schemes
\[
\diagram Y\rto^p \dto^q & X \dto^u \\
         \p(\E^*) \rto^\pi & S
\enddiagram.
\]
Since $u$ is quasi compact it follows $p$ and $q$ are quasi compact
morphisms of schemes. The constructions in the first section of this
paper can be done for arbitrary schemes hence we get a
tautological sequence
\[ 0\rightarrow \O_{\p(\E^*)}(-1) \rightarrow \E \otimes \O_{\p(\E^*)} \]
of sheaves of $\O_{\p(\E^*)}$-modules. On $X$ we have the morphism
\[ \phi:u^*\E\rightarrow \F.\]
Pull these morphisms back to $Y$ to get the composed morphism
\[ \O_{\p(\E^*)}(-1)_Y\rightarrow \E_Y\rightarrow^\phi \F_Y .\]
Let the composed morphism be
\[ \phi_Y:\O_{\p(\E^*)}(-1)_Y\rightarrow \F_Y.\]
\begin{definition} \label{discr} Let $I_1(\phi)=Z(\phi_Y)$ be the \emph{1'st
    incidence scheme} of $\phi$. Let $D_1(\phi)=q(I_1(\phi))$ be the
  \emph{1'st discriminant} of $\phi$.
\end{definition}

Since $q$ is a quasi compact morphism of schemes there is a canonical
scheme structure on $D_1(\phi)$ hence Definition \ref{discr} is well
defined.

By the results of \cite{maa1}, Example 2.15 it follows the
discriminant $D_1(\phi)$
is a simultaneous generalization of the discriminant of a linear
system on a smooth projective scheme, the discriminant of a quasi
compatc morphism of smooth schemes and the classical discriminant of degree d polynomials.

\begin{example} Discriminants of linear systems on projective schemes.  \end{example}

Let $\L\in \Pic(X)$ be a line bundle with $\H^0(X,\L)\neq 0$ and where $X\subseteq \p^n_K$ is a
smooth projective scheme. Let $\pi:X\rightarrow \Spec(K)$ be the structure morphism.
The Taylor morphism for $\L$ is a morphism
\[ T^l:\pi^*\pi_*\L\rightarrow \Pr^l_X(\L) \]
of locally free sheaves. 

\begin{definition} Let $D_l(\L)=D_1(T^l)$ be the \emph{$l$'th
    discriminant of} $\L$.
\end{definition}

We get a subscheme 
\[ D_l(\L) \subseteq \p(\H^0(X,\L)^*), \]
the discriminant of the linear system defined by $\L$. It follows
$D_l(\L)$ is a projective subscheme of $\p(\H^0(X,\L)^*)$. 

If
$X=\p^1_K$ and $\L=\O(d)$ it follows $D_1(\O(d))$ is the classical
discriminant of degree $d$ polynomials as proved earlier in this paper.

\begin{example} The discriminant of a morphism of smooth schemes. \end{example}

Assume $f:U\rightarrow V$ is a quasi compact morphism of smooth schemes and
let
\[ df:f^*\Omega^1_V \rightarrow \Omega^1_U \]
be the differential of $f$. It follows $\Omega^1_U,\Omega^1_V$ are
locally free finite rank sheaves.
The discriminant
\[ D_1(df)\subseteq \p((\Omega^1_V)^*)=\p(T_V) \]
is the discriminant of the morphism $f$. Let $\pi:\p(T_V)\rightarrow
V$ be the projection morphism. 
By \cite{maa1}, Example 2.12 it follows the image scheme $\pi(D_1(df))\subseteq V$ is the
classical discriminant of the morphism $f$. Since $\pi$ is quasi
compact it follows $\pi(D_1(df))$ has a canonical scheme structure.

Dualize Sequence \ref{seq1} to get the exact sequence
\begin{align}\label{seq2}
0\rightarrow \F^* \rightarrow u^*\E^* \rightarrow \Q^* \rightarrow 0.
\end{align}

Take relative projective space bundle to get the closed subscheme
\[ \p(\Q^*)\subseteq \p(u^*\E^*)\cong \p(\E^*)\times_S X=Y.\]

\begin{theorem} \label{lci} The incidence scheme $I_1(\phi)\subseteq
  \p(u^*\E^*)$ is a local complete intersection.
\end{theorem}
\begin{proof} Let $dim(X)=d$ and $dim(S)=l$.
There is an equality
\[ \p(\Q^*)=I_1(\phi) \]
as subschemes of $Y$.  For a proof of this fact see \cite{maa1}, Theorem 2.5.
Since $\Q$ is locally free of rank $e-f$ i to follows
$p:I_1(\phi)\rightarrow X$ is a projective $\p^{e-f-1}$-bundle. It
follows
\[ dim(I_1(\phi))=d+e-f-1.\]
The
sheaf $u^*\E$ is locally free of rank $e$ hence
\[ dim(\p(u^*\E^*))=e+d-1.\]
It follows
\[ codim(I_1(\phi),\p(u^*\E^*))=e+d-1-(d+e-f-1)=f.\]
Let $U=\Spec(A)\subseteq Y$ be an open subscheme where $\O(-1)_Y$ and
$\F_Y$ trivialize.
Restrict the morphism
\[ \tilde{\phi}:\O(-1)_Y\rightarrow \F_Y \]
to $U$ to get the morphism
\[ \tilde{\phi}|_U:Az\rightarrow A\{y_1,..,y_f\}.\]
with 
\[ \tilde{\phi}|_U(z)=b_1y_1+\cdots +b_fy_f\]
where $b_i\in A$.
Let $\I_\phi\subseteq \O_Y$  be the ideal sheaf of $I_1(\phi)$. 
It follows $\I_\phi$ is generated by $\{b_1,..,b_f\}$ on the open
subset $U$. It follows $I_1(\phi)$ is a local complete intersection and
the Theorem is proved.
\end{proof}

By \cite{maa2}, Example 4.5 the morphism
\[ \phi_Y: \O(-1)_Y\rightarrow \F_Y \]
gives rise to a Koszul complex
\begin{align} \label{koszul} 0\rightarrow \O(-f)_Y\otimes \wedge^f\F^*_Y \rightarrow \cdots
\rightarrow \O(-i)_Y\otimes \wedge^i\F^* \rightarrow \cdots
\rightarrow 
\end{align}
\[ \O(-2)_Y\otimes \wedge^2\F^* \rightarrow \O(-1)_Y\otimes \F^*
\rightarrow \O_Y \rightarrow \O_{I_1(\phi)} \rightarrow 0.\]

\begin{definition} \label{incidence} The complex \ref{koszul} is called the
  \emph{incidence complex} of $\phi$.
\end{definition}

Recall the following notions from commutative algebra: Let $A$ be a commutative ring and let
$\lp \subseteq A$ be a prime ideal. Let $ht(\lp)$ be the supremum of
strictly ascending chains of prime ideals
\[ \lp_r\subseteq \cdots \subseteq \lp_1\subseteq \lp_0=\lp \]
ending in $\lp$.
let $coht(\lp)$ be the supremum of strictly ascending chains of prime
ideals
\[ \lp=\lp_0\subseteq \cdots \subseteq \lp_{r-1}\subseteq \lp_r  \]
beginning in $\lp$.
It follows $ht(\lp)=dim(A_\lp)$ and $coht(\lp)=dim(A/\lp)$. We say
$A$ is  \emph{catenary} if $ht(\lp)+coht(\lp)=dim(A)$ for all prime
ideals $\lp$ in $A$. Let for any ideal $I\subseteq A$
$ht(I)=inf\{ht(\lp): I\subseteq \lp\}$.

Let $M$ be an $A$-module. An element $a$ in $A$ is $M$-regular if
$a\neq 0$ and  $ax\neq 0$ for all $0\neq x\in M$. A sequence of
elements $\underline{a}=\{a_1,..,a_k\}$ in $A$ is an
\emph{$M$-sequence} if
the following hold:
\begin{align}
&\label{r1}a_1 \text{ is }M\text{-regular.}\\
&\label{r2}a_{i+1}\text{ is }M/(a_1M+\cdots +a_iM)-\text{regular}\\
&\label{r3}M/(a_1M+\cdots +a_kM)\neq 0
\end{align}

Let in the following $\underline{a}=\{a_1,..,a_k\}$ be a sequence of elements in $A$
and let $\lq=(a_1,..,a_k)$ be the ideal generated by the elements
$a_i$. We say $\underline{a}$ is a \emph{regular sequence in $A$} if it is an $A$-sequence.

Let in the following Proposition $A$ be a Cohen-Macaulay ring.
\begin{proposition} \label{regular} Assume $dim(A/\lq)=dim(A)-k$ and $\lq$ is a prime ideal. Let $\lp\in V(\lq)$ and
  let $\underline{a_\lp}=\{(a_1)_\lp,..,(a_k)_\lp\}$ be the induced
  sequence in $A_\lp$. It follows $\underline{a_\lp}\subseteq
  \lp A_\lp$ is a regular sequence in $A_\lp$.
\end{proposition}
\begin{proof} Since $A$ is Cohen-Macaulay it follows from \cite{mat},
  Theorem 17.9 the quotient $A/\lq$ is catenary. Let $\lp\in \Spec(A/\lq)=V(\lq)$:
We want to calculate $ht(\lq_\lp)$. Since $A$ is Cohen-Macaulay it
follows $A_\lp$ is Cohen-Macaulay, hence it is catenary. It follows
\[
ht(\lq_\lp)=dim((A_\lp)_{\lq_\lp})=dim(A_\lq)=dim(A_\lp)-dim(A_\lp/\lq_\lp).\]
We get
\[ht(\lq_\lp)=dim(A_\lq)=dim(A_\lp)-dim((A/\lq)_\lp)=\]
\[dim(A_\lp)-dim(A/\lq)+dim(A/\lp)=\]
\[dim(A_\lp)-dim(A/\lq)+dim(A)-dim(A_\lp)=\]
\[dim(A)-(dim(A)-k)=k.\]
It follows
\[ ht(\lq_\lp)=k \]
and it follows from \cite{mat}, Theorem 17.4 $\underline{a_\lp}$ is a
regular sequence in $A_\lp$. The Proposition is proved.
\end{proof}

\begin{corollary} \label{resolution} Assume $X$ is an irreducible Cohen-Macaulay scheme. It follows
  the incidence complex \ref{koszul} of $\phi$ is a resolution of $\I_{I_1(\phi)}$.
\end{corollary}
\begin{proof} By Theorem \ref{lci} it follows $I_1(\phi)\subseteq
  \p(u^*\E)$ is a local complete intersection. Hence the ideal sheaf
  $\I_{I_1(\phi)}$ can locally be generated by
  $f=codim(I_1(\phi),\p(u^*\E^*))$ elements.
Let $U=\Spec(A)$ be an open subscheme where
the morphism
\[ \tilde{\phi}: \O(-1)_Y\rightarrow \F_Y \]
trivialize as follows:
\[ \tilde{\phi}|_U:Az\rightarrow A\{y_1,..,y_f\} \]
with 
\[ \tilde{\phi}|_U(z)=b_1y_1+\cdots +b_fy_f.\]
By \cite{hartshorne}, Section II.8 since $A$ is a Cohen-Macaulay ring
the following holds:
Let the sequence
$\underline{b}=\{b_1,..,b_f\}$ generate an ideal $\lq$ in $A$. Since
$I_1(\phi)$ is an irreducible local complete intersection it follows $\lq$ is a prime ideal.
It also follows $dim(A/\lq)=dim(A)-f$. 
Let $\lp \in V(\lq)\subseteq \Spec(A)$. By Proposition \ref{regular}
it follows the sequence $\underline{b_\lp}$ is regular in
$(A/\lq)_\lp$ for all $\lp$.
It follows from \cite{maa2},
Example 4.5 the complex \ref{koszul} is exact since it is locally
isomorphic to the Koszul complex $K_\bullet(\underline{b_\lp})$ on a
regular sequence $\underline{b_\lp}$. The Corollary is proved.
\end{proof}

\begin{definition} Assume $X$ is an irreducible Cohen-Macaulay scheme. Let the
  resolution \ref{koszul} be the \emph{incidence resolution} of $\phi$.
\end{definition}

When we push down the incidence complex \ref{koszul} to $\p(\E^*)$ we
get a double complex with terms given as follows:
\[ \R^jq_*(\O(-i)_Y\otimes \wedge^i\F^*_Y)\cong\]
\[ \R^jq_*(q^*\O(-i)\otimes p^*\wedge^i\F^*)\cong \O(-i)\otimes
\R^jq_*p^*\wedge^i\F^*\cong \]
\begin{align} \label{double}
C^{i,j}(\phi)=\O(-i)\otimes \pi^*\R^ju_*(\wedge^i\F^*).
\end{align}

\begin{definition} Let the double complex $C^{i,j}(\phi)$ from \ref{double} be the
  \emph{discriminant double complex} of $\phi$.
\end{definition}

When $X$ is irreducible Cohen-Macaulay the indicence resolution \ref{koszul} is a
resolution of the ideal sheaf $\I_{I_1(\phi)}$ of the incidence scheme
$I_1(\phi)$. One may ask if the double complex $C^{i,j}(\phi)$ from \ref{double} can be
used to construct a resolution of the ideal sheaf of the discriminant
$D_1(\phi)$. Such a resolution would give a simultaneous resolution of
the ideal sheaf of the discriminant of a linear system on a
smooth projective scheme and the ideal sheaf of the discriminant of a
quasi compact morphism of smooth schemes.

\begin{example} Discriminants of linear systems on flag
  schemes.\end{example}

Let $G$ be a semi simple linear algebraic group over an algebraically
closed field $K$ of characteristic zero and let $P$ in $G$ be a
parabolic subgroup. Let $\pi:G/P\rightarrow \Spec(K)$ be the structure morphism.
Let $\L\in \Pic^G(G/P)$ be a line bundle and
consider the Taylor morphism
\[ T^l:\pi^*\H^0(G/P,\L) \rightarrow \Pr^l(\L) .\]

\begin{definition} Let $I_l(\L(\underline{l}))=I_1(T^l)$ be the
  \emph{$l$'th incidence scheme} of $\L(\underline{l})$.
\end{definition}

We get an incidence scheme
\[ I_l(\L)\subseteq \p(\H^0(G/P,\L)^*)\times G/P=Y \]
and an incidence complex
\begin{align} \label{flag}
0\rightarrow \O(-r)_Y\otimes \wedge^r\Pr^l(\L)_Y^*\rightarrow \cdots
\rightarrow \O(-2)_Y\otimes \wedge^2\Pr^l(\L)_Y^* \rightarrow 
\end{align}
\[ \O(-1)_Y\otimes \Pr^l(\L)^*_Y \rightarrow \O_Y \rightarrow
\O_{I_l(\L)}\rightarrow 0. \]
Since $G/P$ is smooth it is Cohen-Macaulay. Furthermore $G/P$ is
irreducible hence by Corollary \ref{resolution} it follows the incidence
complex \ref{flag} is a resolution of the ideal sheaf of
$I_l(\L)$. When we push the incidence resolution \ref{flag} down to
$\p(\H^0(G/P,\L)^*)$ we get a double complex with terms
\[ C^{i,j}(T^l)=\O(-i)\otimes \H^j(G/P,\wedge^i\Pr^l(\L)^*).\]
If one can calculate the higher cohomology groups 
\[ \H^j(G/P,\wedge^i \Pr^l(\L)^*) \]
for all $i,j$ one can decide if the discriminant double complex
$C^{i,j}(T^l)$ gives information on a resolution on the ideal sheaf of
$D_l(\L)$. 

\begin{definition}  Let $C^{i,j}_l(\L(\underline{l}))=C^{i,j}(T^l)$ be
  the \emph{$l$'th discriminant double complex} of $\L(\underline{l})$.
\end{definition}

Note: By \cite{maa1} the discriminant $D_1(\O(d))$ on $\p^1$ is a
determinantal scheme hence by the results of \cite{lascoux} one gets
information on its resolutions. If one can prove a class of
discriminants are determinantal schemes one get two approaces to the
study of resolutions: One via jet bundles, Taylor morphisms and the
$l$'th discriminant double
complex $C^{i,j}_l(\L(\underline{l} ))$ from \ref{double}. 
Another one via determinantal schemes and the construction in \cite{lascoux}.

\begin{example} Canonical filtrations of irreducible $\SL(E)$-modules. \end{example}

On projective space the cohomology group
$\H^j(G/P,\wedge^i \Pr^l(\L)^*)$ is completely determined (see \cite{maa2},
Theorem 4.10) since the structure
of the jet bundle is classified. It remains to give a similar
description of the structure of the jet bundle on grassmannians and
flag schemes. There is work in progress on this problem (see
\cite{maa3},\cite{maa4} and \cite{maa5}).

In \cite{maa5} we prove in Theorem 3.10 the following result: Let
$G=\SL(E)$ be the special linear group on $E$ where $E$ is a finite
dimensional vector space over an algebraically closed field $K$ of
characteristic zero. Let $V_\lambda$ be a finite dimensional
irreducible $G$-module with highest weight vector $v$ and highest
weight $\lambda=\sum_{i=1}^k l_i\omega_{n_i}$. Here $l_i\geq 1$ and
$\underline{l}=(l_1,..,l_k)$. Let $P$ in $G$ be the parabolic subgroup
stabilizing $v$. It follows there is an isomorphism
\[ \Pr^l(\L(\underline{l}))(\overline{e})^*\cong \U(\sl(E))v \]
of $P$-modules where $\overline{e}\in G/P$ is the class of the
identity.
Here $\L(\underline{l})\in \Pic^G(G/P)$ is the line bundle with
$\H^0(G/P,\L(\underline{l}))^*\cong V_\lambda$. The $P$-module
$\U(\sl(E))v \subseteq V_\lambda$ is the $l$'th piece of the canonical
filtration of $V_\lambda$ as studied in \cite{maa3}. It is hoped such a description of
$\Pr^l(\L(\underline{l}))$ will give information on the cohomology
group
\[ \H^j(G/P,\wedge^i\Pr^l(\L(\underline{l}))^*) \]
for all $i,j$.

Note: In the paper \cite{maa3} a complement of the $l$'th piece of the
filtration of the annihilator ideal 
\[ ann_l(v) \subseteq \U_l(\sl(E)) \]
of the highest weight vector $v$ in $V_\lambda$ is calculated. It is
given by the $l$'th piece of the canonical filtration of the universal
enveloping algebra of a sub Lie algebra $\ln(\underline{n})\subseteq
\sl(E)$. The Lie algebra $\ln(\underline{n})$ is canonically
determined by a flag $E_\bullet$ in $E$ determined by the highest
weight $\lambda$ for $V_\lambda$. A basis for $E$ compatible with the
flag $E_\bullet$ gives a canonical basis for the $P$-module
$\U_l(\sl(E))v$. When $\U_l(\sl(E))v=V_\lambda$ this gives a canonical
basis for the $\SL(E)$-module $V_\lambda$ defined in terms of the
universal enveloping algebra $\U(\sl(E))$. 

It is hoped knowledge on
the canonical filtration $\U_l(\sl(E))v$ as $P$-module will give
information on the problem of calculating the cohomology group
\[ \H^i(G/P, \wedge^j\Pr^l(\L(\underline{l}))^*) \]
for all $i,j\geq 0$. Such a result will as explained above be used in
the study of resolutions of ideal sheaves of discriminants of linear
systems on flag schemes. The main aim is to give a resolution of the
ideal sheaf of $D_l(\L(\underline{l}))$ for any $l\geq 1$ and
$\L(\underline{l})\in \Pic^{\SL(E)}(\SL(E)/P)$.
We get an approach to the study of resolutions of ideal sheaves of
discriminants using algebraic groups, canonical filtrations, the
theory of highest weights, higher direct images of sheaves and the
discriminant double complex. This approach will be used in future
papers on the subject (see \cite{maa6}).

\section{Discriminants of linear systems on projective space}

In this section we study jet bundles and discriminants of linear
systems on projective space. Let $\p(V^*)$ be projective space
parametrizing lines in a fixed $K$-vector space $V$ of dimension $n+1$
and let $\O(d)=\O(1)^{\otimes d}$ be the $d$'th tensor product of the
tautological quotient bundle as constructed in Section 2. Let
$W=\H^0(\p(V^*),\O(d))$ be the vector space of global sections of
$\O(d)$. In this section we construct local generators for the ideal
sheaf of the $l$'th incidence scheme $I_l(\O(d))$ for all integers $1\leq l
\leq d$. The aim of the construction is to use it to study the projection
morphism
\[ \pi:I_l(\O(d))\rightarrow D_l(\O(d)) \]
from the incidence scheme to the discriminant $D_l(\O(d))$.

Let $V=K\{e_0,..,e_n\}$ and let $V^*=K\{x_0,..,x_n\}$. Let
$S=\sym_K(V^*)=K[x_0,..,x_n]$
and let $\p(V^*)=\Proj(S)$. Let
$I=(i_1,..,i_k)$ with $i_j\geq 0$ for all $j$. Let $\#I=\sum i_j$ and
let $I!=i_1!\cdots i_k!$.
Let $u_1,..,u_k$ be a set of independent variables over $K$. and let
\[ \partial^I_U=\puI .\]
It follows
\[ \partial^I_U\in \Diff_K(K[u_1,..,u_k]) \]
is a differential operator of order $\#I$. Let $p=(p_1,..,p_k)$ and
let 
$U^p=u_1^{p_1}\cdots
u_k^{p_k}$
\begin{lemma} \label{partial} The following formula holds for all
  integers $0 \leq i_j\leq p_j$:
\[ \frac{1}{I!}\partial^I_U(u_1^{p_1}\cdots u_k^{p_k})=
\binom{p_1}{i_1}\binom{p_2}{i_2}\cdots
\binom{p_k}{i_k}u_1^{p_1-i_1}u_2^{p_2-i_2}\cdots u_k^{p_k-i_k} \]
\end{lemma}
\begin{proof} The proof is straight forward.
\end{proof}

\begin{lemma} \label{taylor} Consider $f(u_1,..,u_k)=u_1^{p_1}\cdots
  u_k^{p_k}$ and let $du_1,..,du_k$ be indepentent variables over
the ring  $K[u_1,..,u_k]$. Let $\# p=\sum_j p_j$.
There is an equality
\[
f(u_1+du_1,..,u_k+du_k)=\sum_{m=0}^{\#
  p}\sum_{\#I=m}\frac{\partial^I_U(f)}{ I!}du_1^{i_1}\cdots du_k^{i_k} \]
in the ring $K[t_0,..,t_n][du_1,..,du_k]$.
\end{lemma}
\begin{proof} From Lemma \ref{partial} the following calculation
  holds:
\[ f(u_1+du_1,..,u_k+du_k)=\]
\[(u_1+du_1)^{p_1}\cdots (u_k+du_k)^{p_k}=\]
\[(\sum_{i_1=0}^{p_1}\binom{p_1}{i_1}u_1^{p_1-i_1}du_1^{i_1} )\cdots 
(\sum_{i_k=0}^{p_k}\binom{p_k}{i_k}u_k^{p_k-i_k} du_k^{i_k})=\]
\[ \sum_{m=0}^{\# p}\sum_{\#I=m}\binom{p_1}{i_k}\cdots
\binom{p_k}{i_k}u_1^{p_1-i_1}\cdots u_k^{p_k-i_k}du_1^{i_1}\cdots du_k^{i_k} =\]
\[ \sum_{m=0}^{\# p}\sum_{\#I=m}\frac{\partial^I_U(f)}{ I!}du_1^{i_1}\cdots du_k^{i_k} \]
and the Lemma is proved.
\end{proof}

Define the following map:
\[ T:K[u_1,..,u_k]\rightarrow K[u_1,..,u_k][du_1,..,du_k] \]
by
\[ T(f(u_1,..,u_k))=f(u_1+du_1,..,u_k+du_k).\]

\begin{proposition} \label{taylor} The following formula holds:
\[ T(f)= \sum_{m=0}^{deg(f)}\sum_{\#I=m}\frac{\partial^I_U(f)}{\#
  I!}du_1^{i_1}\cdots du_k^{i_k} \]
\end{proposition}
\begin{proof} The Proposition follows from Lemma \ref{taylor} since
  $f$ is a sum of monomials in the variables $u_1,..,u_k$.
\end{proof}

The map $T$ is the \emph{formal Taylor expansion} of the polynomial
$f(u_1,..,u_k)$ in the variables $du_1,..,du_k$. 
It follows $T\in \Diff_K(K[t_0,..,t_n], K[t_0,..,t_n][du_1,..,du_k])$.

Let $\O(d)=\O(1)^{\otimes d}$ where $\O(1)$ is the tautological
quotient bundle from section two and let $W=\H^0(\p(V^*),\O(d))$. Let
$W$ have basis 
\[ B=\{ x_0^{p_0}\cdots x_n^{p_n}:\sum p_i=d\}.\]
We write $x_0^{p_0}\cdots x_n^{p_n}=s^p$ with $p=(p_0,..,p_n)$. Write
$\#p=\sum p_i$ and $p!=p_0!\cdots p_n!$. Let $W^*$ have basis
\[ B^*=\{(x_0^{p_0}\cdots x_n^{p_n})^*: \#p=d\}.\]
Write $(x_0^{p_0}\cdots x_n^{p_n})^*=y^p$. 
It follows $\sym_K(W^*)=K[y^p:\#p=d]$ is the polynomial ring on the
independent variables $y^p$. Let $Y=\p(W^*)\times \p(V^*)$. We get a
diagram
\[
\diagram Y\rto^p \dto^q & \p(V^*) \dto^\pi \\
          \p(W^*) \rto^\pi & \Spec(K)
\enddiagram.
\]
Let $D(y^p)\times D(x_i)\subseteq Y$ be the basic open subset where
$y^p$ and $x_i$ are non zero.
On $\p(W^*)$ there is the tautological subbundle
\[ \O(-1)\subseteq W\otimes \O_{\p(W^*)} \]
from section two.
On projective space $\p(V^*)$ there is the Taylor morphism
\[ T^l: W\otimes \O_{\p(V^*)}\rightarrow \Pr^l(\O(d)) .\]
Pull these maps back to $Y$ via $p$ and $q$ to get the composed
morphism
\[ T^l_Y:\O(-1)_Y\rightarrow \Pr^l(\O(d))_Y.\]
By definition $I_l(\O(d))=Z(T^l_Y)$ is the zero scheme of $T^l_Y$. We
get a diagram of maps of schemes
\[
\diagram I_l(\O(d)) \rto^i \dto^{\tilde{q}} & Y\rto^p \dto^q & \p(V^*) \dto^\pi \\
         D_l(\O(d)) \rto^j &           \p(W^*) \rto^\pi & \Spec(K)
\enddiagram.
\]
where $i$ and $j$ are closed immersions of schemes. Since $I_l(\O(d))$
is a closed subscheme of $Y$ and $q$ is a proper morphism it follows
$D_l(\O(d))$ is a closed subscheme of $\p(W^*)$. Since $\p(V^*)$ is
smooth it is Cohen-Macaulay. Also $\p(V^*)$ is irreducible hence the incidence complex 
\[ 0\rightarrow \O(-r)\otimes \wedge^r \Pr^l(\O(d))_Y \rightarrow
\cdots \rightarrow \]
\[ \O(-i)\otimes \wedge\Pr^l(\O(d))_Y \rightarrow \cdots \rightarrow
\O(-1)\otimes \Pr^l(\O(d))_Y\rightarrow \O_Y\rightarrow
\O_{I_l(\O(d))}\rightarrow 0\]
from Corollary \ref{resolution} is a resolution of the ideal sheaf of
$I_l(\O(d))$. The aim of this
section is to calculate local generators for the ideal sheaf of
$I_l(\O(d))$ with respect to the open cover $\{ D(y^p)\times D(x_i):
\# p=d, i=0,..,n\}$ of $Y$.
The tautological subbundle
\[ \O(-1)\rightarrow W\otimes \O_{\p(W^*)} \]
is the sheaffification of the following sequence:
\[ \alpha:K[y^p:\# p=d](-1) \rightarrow K[y^p:\# p=d]\otimes \{s^p:\# p=d\}\]
defined by
\[ \alpha(1)=\sum_{p} y^p\otimes s^p=\sum_p (s^p)^*\otimes s^p.\]
Consider the open subset $D(y^p)\subseteq \p(W^*)$. Let
$u^q=\frac{y^q}{y^p}$ be coordinates on $D(y^p)$. It follows the
coordinate ring $\O(D(y^p))$ is the polynomial ring
\[ K[u^q:\# q=d].\]
Consider the open set $D(x_i)$ and let $t_k=\frac{x_k}{x_i}$ be
coordinates on $D(x_i)$. It follows the coordinate ring $\O(D(x_i))$ is
the polynomial ring $K[t_0,..,t_n]$. The coordinate ring of
$D(y^p)\times D(x_i)$ is the polynomial ring 
\[ \O(D(y^p)\times D(x_i))=K[t_0,..,t_n][y^p: \# p=d].\]

The Taylor map
\[T^l: \O_{\p(V^*)}\otimes W \rightarrow \Pr^l(\O(d)) \]
is defined as follows: Let $U_i=D(x_i)\subseteq \p(V^*)$ be the basic
open subset where $x_i\neq 0$. We get a map
\[T^l_{U_i}:K[t_0,..,t_n]\otimes \{s^p:\# p=d\} \rightarrow
K[t_0,..,t_n]\{dt_0^{i_0}\cdots dt_n^{i_n}\otimes x_i^d :i_0+\cdots
+i_n\leq l\} \]
which we now make explicit.
Let
\[ s^p=x_0^{p_0}\cdots x_n^{p_n} \]
with $p_0+\cdots +p_n=d$. It follows 
\[ d_i=d-p_0-\cdots -p_n.\]
We may write
\[ s^p=x_0^{p_0}\cdots x_n^{p_n}=t_0^{p_0}\cdots t_n^{p_n}x_i^d\]
in $K(x_0,..,x_n)$. Note: $t_i=x_i/x_i=1$. We get
\[ T^l_{U_i}(s^p)=1\otimes t_0^{p_0}\cdots t_n^{p_n}x_i^d=\]
\[(t_0+dt_0)^{p_0}\cdots (t_n+dt_n)^{p_n}\otimes x_i^d.\]
Let $f_p=t_0^{p_0}\cdots t_n^{p_n}$. It follows
\[ T^l_{U_i}(s^p)=f_p(t_0+dt_0,..,t_n+dt_n)\otimes x_i^d \in
\Pr^l(\O(d))(U_i).\]
Here
\[ \Pr^l(\O(d))(U_i)=K[t_0,..,t_n]\{ dt_0^{i_0}\cdots
dt_n^{i_n}\otimes x_i^d:\sum i_j\leq l\}. \]
Let $I=(i_0,..,i_n)$ with $i_j\in \mathbf{Z}$ integers.
Let 
\[ \partial^I_T=\ptI \]
where $\frac{\partial}{\partial_{t_i}}$ is partial derivative with
respect to the $t_i$-variable. It follows 
\[ \partial^I_T\in \Diff_K(K[t_0,..,t_n]) \]
is a differential operator of order $\# I$.

\begin{lemma} \label{taylor1} The following holds:
\[ T^l_{U_i}(s^p)=\sum_{k=0}^{d-p_i}\sum_{\#
  I=k}\frac{1}{I!}\partial^I_T(f_p)dt_0^{i_0}\cdots dt_n^{i_n}\otimes x_i^d. \]
\end{lemma}
\begin{proof} By the discussion above it follows
\[ T^l_{U_i}(s^p)=f_p(t_0+dt_0,..,t_n+dt_n)\otimes x_i^d.\]
The Lemma follows from  Proposition \ref{taylor}.
\end{proof}

Consider the open set 
\[ U_{p,i}=D(y^p)\times D(x_i)\subseteq \p(W^*)\times \p(V^*)=Y .\]
The morphism $\alpha$ looks as follows on $U_{p,i}$:
\[\alpha(\frac{1}{y^p})=\sum_{\# q=d}u^q \otimes s^q=\sum_{\#
  q=d}u^q\otimes x_0^{q_0}\cdots x_n^{q_n}.\]
Let $f_q=t_0^{q_0}\cdots t_n^{q_n}$ with $t_i=x_i/x_i=1$. Let 
\[ f=\sum_{\# q=d}u^qf_q=\sum_{\# q=d}u^qt_0^{q_0}\cdots t_n^{q_n}.\]
We want to calculate the expression
\[ T^l_{U_{p,i}}(\alpha(\frac{1}{y^p}))= \sum_{\# q=d}u^q
\otimes T^l_{U_{p,i}}(s^q)\]
as an element of the module
\[ \Pr^l(\O(d))(U_{p,i})=K[t_0,..,t_n][y^p:\# p=d]\{ dt_0^{q_0}\cdots
dt_n^{q_n}\otimes x_i^d :0\leq \# q \leq l\}.\]

\begin{proposition} \label{maintaylor}  The following holds:
\[ T^l_{U_{p,i}}(\alpha(\frac{1}{y^p}))=\sum_{k=0}^{l}\sum_{\#
  I=k}\frac{1}{I!}\partial^{\# I}_T(f)dt_0^{i_0}\cdots
dt_n^{i_n}\otimes x_i^d\]
in $\Pr^l(\O(d))(U_{p,i})$.
\end{proposition}
\begin{proof} Let $T^l_{p,i}=T^l_{U_{p,i}}$.
By Lemma \ref{taylor1} we get the following calculation:
\[ T^l_{p,i}(\sum_{\# q=d}u^q \otimes s^q)= \sum_{\# q=d}u^q T^l_{p,i}(s^q)=\]
\[ \sum_{\# q=d}u^q T^l_{p,i}(s^q)=\sum_{\#
  q=d}u^q\sum_{k=0}^{deg(f_q)}\sum_{\# I=k}\frac{1}{I!}\partial^{\#
  I}_T(f_q)dt_0^{i_0}\cdots dt_n^{i_n}=\]
\[ \sum_{k=0}^{deg(f)}\sum_{\# I=k}\frac{1}{I!}\partial^{\#
  I}_T(f)dt_0^{i_0}\cdots dt_n^{i_n}= \sum_{k=0}^{l }\sum_{\#
  I=k}\frac{1}{I!}\partial^{\#
  I}_T(f)dt_0^{i_0}\cdots dt_n^{i_n} \]
and the Proposition is proved.
\end{proof}

\begin{corollary} The ideal sheaf $\I$ of $I_l(\O(d))$ is on
  $U_{p,i}=D(y^p)\times D(x_i)$ generated by the following set
\[ \I_{U_{p,i}}=\{ \frac{1}{I!}\partial^{\# I}_T(f) : \# I=k, k=0,..,l\}.\]
as $K[t_0,..,t_n][y^p:\# p=d]$-module.
\end{corollary}
\begin{proof} The Corollary follows from Proposition \ref{maintaylor}.
\end{proof}

A section $s\in \H^0(\p(V^*),\O(d))$ is a homogeneous polynomial in
$x_0,..,x_n$ of degree $d$. If we restrict $s$ to $D(x_i)$ we get a
section $s=f(t_0,..,t_n)x_i^d$ where $f$ is a polynomial of degree $d$
in $t_0,..,t_n$. The ideal sheaf $\I$ of $I_l(\O(d))$ is on
$D(y^p)\times D(x_i)$ generated as $K[t_0,..,t_n][y^p:\# p=d]$-module
by all possible partial derivatives of degree $m \leq l$ of the polynomial $f$.
The aim of this calculation is to use it to describe the projection morphism
\[ \tilde{q}:I_l(\O(d))\rightarrow D_l(\O(d)) .\]
We want to calculate its fiber, the dimension of its fiber and the
dimension of $D_l(\O(d))$. The morphism $\tilde{q}$ is generically
smooth, the schemes $I_l(\O(d))$ and $D_l(\O(d))$ are irreducible
hence if we know the dimension of a generic fiber $\tilde{q}^{-1}(z)$
of $\tilde{q}$ we can calculate $dim(D_l(\O(d)))$.

\section{Appendix: Jet bundles and Taylor morphisms}

In this section we prove some general properties of jet bundles and
Taylor morphisms. We show how the Taylor morphism at each $K$-rational
point formally taylor expands a section of a locally free sheaf.

Let $K$ be a fixed basefield of characteristic zero and let $X$ be a
scheme of finite type over $K$. This means $X$ has an open affine
cover $U=\{ \Spec(A_i)\}_{i\in I}$ with $A_i$ a finitely generated
$K$-algebra for all $i\in I$. By $X(K)$ we mean the set of
$K$-rational points of $X$. If $X=\Spec(A)$ is an affine scheme a
$K$-rational point $\lm\in X$ corresponds to a maximal ideal $\lm
\subseteq A$ with residue field $\kappa(\lm)=K$. Let $\E$ be a locally
free finite rank $\O_X$-module and let $\Pr^l_X(\E)$ be the $l$'th jet
bundle of $\E$.

\begin{lemma} Let $x \in X(K)$. There is for all $l \geq 0$ an isomorphism of
  $\kappa(x)$-vector spaces
\[ \Pr^l_X(\E)(x)\cong \E_{x}/\lm_x^{l+1}\E_x.\]
\end{lemma}
\begin{proof} Assume $x\in \Spec(A)\subseteq X$ is an open affine
  subset of $X$ and let $\lm=\lm_x\subseteq A$ be the maximal ideal
  corresponding to $x$. Let $I\subseteq A\otimes_K A$ be the ideal of
  the diagonal and let $U=\Spec(A)$. It follows the restricted sheaf
$\Pr^l_X(\E)|_U$ equals the sheaffification of the $A$-module
$P^l_A(E)=A\otimes_K A/I^{l+1}\otimes E$ where the sheafification of $E$ of $U$
equals $\E|_U$. The $A$-module $E$ is locally free of finite rank.
There is an exact sequence
\begin{align}\label{exact}
 0\rightarrow I^{l+1}\rightarrow A\otimes_K A \rightarrow P^l_A
\rightarrow 0
\end{align}
of left $A$-modules. We may write $A=K[t_1,..,t_n]/J$ where $J$ is an
ideal in $K[t_1,..,t_n]$. Write
\[ A\otimes_K A\cong K[t_1,..,t_n]/J\otimes_K K[t_1,..,t_n]/J\cong \]
\[ K[t_1,..,t_n,u_1,..,u_n]/(J_t,J_u) .\]
It follows the ideal $I$ equals
\[ I=(t_1-u_1,..,t_n-u_n).\]
There is an isomorphism $A/\lm\cong K$ hence the maximal ideal $\lm$
equals $(t_1-a_1,..,t_n-a_n)$ with $a_i\in K=A/\lm$. 
Tensor the exact sequence \ref{exact} with $-\otimes_A A/\lm$ to get
\[ I^{l+1}\otimes_A A/\lm \rightarrow A\otimes_K A\otimes_A A/\lm
\rightarrow P^l_A\otimes_A A/\lm \rightarrow 0.\]
There is an isomorphism 
\[ A\otimes_K A\otimes_A A/\lm\cong A\otimes_K A/\lm\cong A\otimes_K
K\cong A \]
hence we get a map
\[ \phi:I^{l+1}\otimes_A A/\lm \rightarrow A.\]
When we tensor $I^{l+1}$ with $A/\lm$ we get $\phi(u_i)=a_i\in
K\subseteq A$. It follows $\psi(I^{l+1}\otimes_A K)=\lm^{l+1}\subseteq
A$. Since $\lm$ is a maximal ideal it follows $P^l_A\otimes_A
A/\lm\cong A/_\lm^{l+1}\cong \O_{X,x}/\lm_x^{l+1}$. 
We have proved there is an isomorphism
\[ \Pr^l_X(x)\cong \O_{X,x}/\lm_x^{l+1}\]
of $\kappa(x)$-vector spaces.
We get
\[ \Pr^l_X(\E)(x)\cong \Pr^l_X(\E)|_U(x)\cong P^l_A(E)\otimes_A
\kappa(x)\cong\]
\[ P^l_A\otimes_A E\otimes_A A/\lm \cong P^l_A\otimes_A A/\lm
\otimes_A E \cong A/\lm^{l+1}\otimes_A E \cong E/\lm^{l+1}E 
\cong
\E_x/\lm_x^{l+1}\E_x\]
and the Lemma is proved. 
\end{proof}

Let $s\in \H^0(X,\E)$ be a global section. We may for any point $x\in
X(K)$ restrict $s$ to the stalk $\E_x$. We get a canonical $K$-linear map
\[ \H^0(X,\E)\rightarrow \E_x/\lm_x^{l+1}\E_x .\]
In the following we make the identification $\Pr^l_X(\E)(x)\cong \E_{x}/\lm_x^{l+1}\E_x$.
Recall there is a Taylor morphism
\[ T^l:\H^0(X,\E)\otimes \O_X \rightarrow \Pr^l_X(\E) \] 
for all $l\geq 1$. When $x\in X(K)$ we get a map of $K$-vector spaces
\[ T^l(x):\H^0(X,\E)\rightarrow \Pr^l_X(\E)(x)\cong \E_x/\lm_x^{l+1}\E_x.\]

\begin{lemma} The Taylor map $T^l(x)$ equals the canonical map
\[ T^l(x):\H^0(X,\E)\rightarrow \E_x/\lm_x^{l+1}\E_x .\]
\end{lemma}
\begin{proof} The proof is an exercise.
\end{proof}

\begin{example} Formal Taylor expansion of sections of
  sheaves.\end{example}

Assume $X=\Spec(A)$ where $A=K[t_1,..,t_n]/J$ and let $x$ be a
$K$-rational point with maximal ideal
$\lm=(t_1-a_1,..,t_n-a_n)$. Let $\E=\O_X$.
We get a taylor map
\[ T^l:\H^0(X,\E)\otimes \O_X \rightarrow \Pr^l_X(\E) .\]
Take the fiber of $T^l$ at $x$ to get the map
\[ T^l(x):\H^0(X,\E)\rightarrow \Pr^l_X(\E)(x) .\]
It follows we get a map
\[ T^l(x):A\rightarrow \Pr^l_X(\E)(x)\cong A/\lm^{l+1} \]
Let $I=(i_1,..,i_n)$ be a tuple of integers with $i_k\geq 0$. 
Let 
\[ \partial^{\# I}_T=\pttI .\]

\begin{proposition} \label{formal} Let $a=\overline{f(t)}\in A=K[t_1,..,t_n]/J$.
The map $T^l(x)$ looks as follows:
\[ T^l(x)(a)=\sum_{m=0}^l\sum_{i_1+\cdots +i_n=m}\frac{\partial^{\#
    I}_T f(a_1,..,a_n)}{I!}(t_1-a_1)^{i_1}\cdots (t_n-a_n)^{i_n}. \]
\end{proposition}
\begin{proof} There is an inclusion of ideals $J\subseteq \lm$
The $l$'th Taylor morphism is a morphism
\[ T^l(x):A \rightarrow A/\lm^{l+1} .\]
We get a well defined map
\[ T^l(x):K[t_1,..,t-n]/J\rightarrow K[t_1,..,t_n]/(t_1-a_1,..,t_n-a_n)^{l+1} .\]
We may for any polynomial $a=\overline{f(t_1,..,t_n)}$ write
\[ T^l(x)(a)= f(t_1,..,t_n)\text{ mod }\lm^{l+1}\]
in the ring $K[t_1,..,t_n]/\lm^{l+1}$.
We may write
\[ T^l(x)(a)=f(a_1+t_1-a_1,..,a_n+t_n-a_n).\]
We get from Proposition \ref{taylor}
\[ T^l(x)(a)= \sum_{m=0}^l\sum_{i_1+\cdots +i_n=m}\frac{\partial^{\#
    I}_T f(a_1,..,a_n)}{I!}(t_1-a_1)^{i_1}\cdots (t_n-a_n)^{i_n}\text{ mod
  }\lm^{l+1}\]
and the Proposition is proved.
\end{proof}

Assume $X=\Spec(A)$ is smooth over $K$.
For $a\in A$ and $x\in X(K)$ with maximal ideal
$\lm=(t_1-a_1,..,t_n-a_n)$ the element $T^l(x)(a)\in \Pr^l_X(x)$
from Proposition \ref{formal} is the \emph{formal Taylor expansion} of
$a$ at the point $(a_1,..,a_n)\in K^n$. If $K=\mathbf{C}$ is the field
of complex numbers we see this equals the classical Taylor expansion
of $a$ in $(a_1,..,a_n)$ viewed as holomorphic function on the complex
algebraic manifold $X$. Hence if $\E=\O_X$ the Taylor morphism
\[ T^l:\H^0(X,\O_X)\otimes \O_X \rightarrow \Pr^l_X(\O_X) \]
calculates at each point $x\in X(K)$ the formal Taylor expansion
\[ T^l(x)(a)\in \Pr^l_X(\O_X)(x) \cong \O_{X,x}/\lm_x^{l+1}\]
of a global section $a\in \H^0(X,\O_X)$.

This paper form part of a series of papers where the aim is to study
different properties of discriminants of morphisms of locally free sheaves
We want  to describe the singularities, to calculate dimensions and
degrees, to describe syzygies and to resolve singularities for a class
of discriminants. Some work has been done in this direction (see
\cite{maa1}, \cite{maa2}, \cite{maa3}, \cite{maa4},\cite{maa5} and
\cite{maa6} for some recent preprints on this subject). 
Since discriminants of linear systems on flag schemes are
closely related to jet bundles a study of jet bundles of line bundles
on flag schemes and representations of algebraic groups 
has been done in the papers \cite{maa4} and \cite{maa5}. 
In \cite{maa5} we classify 
the $P$-module $\Pr^l(\L)(\overline{e})^*$ and relate it to the \emph{canonical
  filtration} of $V_\lambda$ as studied in \cite{maa3}. The aim is to
use this classification to calculate the cohomology group
$\H^i(G/P,\wedge^j \Pr^l(\L)^*)$ for all integers $i,j\geq 0$.
Such a calculation will give results on syzygies of discriminants of
linear systems on flag schemes as proved in the previous section.

\end{document}